\documentclass[12pt]{article}
\usepackage{amsmath, amsfonts, amssymb, amsthm}
\usepackage[colorlinks = true, linkcolor = blue]{hyperref}
\newtheorem{theorem}{Theorem}[section]
\newtheorem{proposition}[theorem]{Proposition}

\newtheorem{remark}[theorem]{Remark}

\newtheorem{example}[theorem]{Example}

\newtheorem{corollary}[theorem]{Corollary}
\newtheorem{lemma}[theorem]{Lemma}

\def\R{\mathbb R} \def\Z{\mathbb Z} 

\def\N{\mathbb N}

\def\Q{\mathbb Q}

\def\<{\,<\!}
\def\>{\!>\,}

\def\R{\mathbb{R}}

\def\F{\mathbb{F}}
\def\E{\mathbb{E}}
\def\N{\mathbb{N}}

\def\Ord{{\rm Ord}}

\def \Z {\mathbb{Z}}
\def \R {\mathbb{R}}

\begin{document}

\title{Group actions and power maps for groups over non-Archimedean local fields}

\author{Arunava Mandal and C.R.E. Raja}

 \maketitle

 

\begin{abstract}
We consider linear groups and Lie groups over a non-Archimedean local field $\F$ for which the power map $x\mapsto x^k$ has a dense image or it is surjective.  We prove that the group of $\F$-points of such algebraic groups is a compact extension of unipotent groups with the order of the compact group being relatively prime to $k$.  This in particular shows that the power map is surjective for all $k$ is possible only when the group is unipotent or trivial depending on whether the characteristic of $\F$ is zero or positive.  Similar results are proved for Lie groups via the adjoint representation.  To a large extent, these results are extended to linear groups over local fields and global fields.

\end{abstract}


\section{Introduction}
We will be considering existence of the solution to the equation
$$x^k =g, ~~k >1 \eqno(1)$$ in a group $G$ for every $g\in G$.
Our aim is to find structural conditions that are equivalent to
equation (1) having a solution in $G$ for every $g\in G$.  We first
note that equation (1) having a solution in $G$ for every $g\in G$
is equivalent to the corresponding power map $P_k\colon G \to G$
defined by $P_k(g) = g^k$ being surjective.  Thus, we pay attention
surjectivity of the power maps. For a (topological) group $G$ and
$k\geq 1$, we say that $P_k$ is surjective on $G$ (resp. dense in
$G$) if $P_k\colon G \to G$ defined by $P_k(x)=x^k$ for all $x\in G$
is surjective (resp. has dense image).

Surjectivity of power maps have been studied for algebraic groups
over local fields and algebraically closed fields of characteristic
zero in \cite{Ch1}-\cite{Ch3} and for semisimple algebraic groups
over algebraically closed fields in \cite{St}: see also the
references cited therein \cite{Ch3}. Certain class of solvable (not
necessarily algebraic) group is considered in \cite{DaM}.

Density of the power map has also been considered for connected Lie
groups in \cite{BhM}, \cite{Ma1}, \cite{MaS} and for real algebraic
groups in \cite{Ma2}.

We consider groups that have linear representation over
non-Archimedean local fields and our approach involves linear
dynamics and tidy subgroups: see \cite{CoG} for linear dynamics and
\cite{Wi0} for tidy subgroups.

Let $\F$ be a field, and let $G$ be a smooth affine $\F$-group. Then
the $\F$-unipotent radical $R_{u,\F}(G)$ of $G$ is defined to be the
largest smooth connected normal unipotent $\F$-subgroup of $G$.  We
say that $G$ is pseudo-reductive $\F$-group if $R_{u,\F }(G) =\{e\}$
where $e$ is the identity element of $G$.  An $\F$-split unipotent
radical $R_{us, \F}(G)$ is a maximal smooth connected $\F$-split
unipotent normal subgroup of $G$ and $G$ is said to be
quasi-reductive if $R_{us,\F}(G)=\{e\}$.  Note that
$R_{us,\F}(G)\subset R_{u, \F}(G)$. We equip $G(\F )$ with topology
inherited from $\F$ when $\F$ is a local field.  If $\F$ is a
non-Archimedean local field, then $G(\F)$ is totally disconnected
locally compact group.









It is well-known that if $G$ is a smooth connected affine $\F$-group
then $G/R_{u,\F}(G)$ is a pseudo-reductive $\F$-group  and
$G/R_{us,F}(G)$ is a quasi-reductive $\F$- group (see \cite{CGP}
Corollary B.3.5).


We first prove surjectivity and density are equivalent for $G(\F )$
and characterize the surjectivity/density of the $k$-th power map
$P_k$ for $G(\F )$ in terms of the $\F$-points of the
quasi-reductive quotient group $G/R_{us,\F}(G)$.  We recall that the
canonical map from $G(\F)/R_{us,\F}(G)(\F)$ into $
(G/R_{us,\F}(G))(\F) $ is an isomorphism as $H^1(\F, R_{us,\F}(G))=0$.

\begin{theorem}\label{T2}
Let $\F$ be a non-Archimedean local field. Let $G$ be an $\F$-group,
and $R_{us,\F}(G)$ be the $\F$-split unipotent radical of $G$.
Suppose that the characteristic of $\F$ does not divide $k$.  Then
the following are equivalent:
    \begin{enumerate}
    \item[{$(a)$}] $P_k$ is dense in $G(\F)$;
    \item[{$(b)$}] $G(\F)/R_{us, \F}(G)(\mathbb F)$ is compact and
    $P_k$ is surjective on $G(\F)/R_{us,\F}(G)(\F)$;

    \item[$(c)$] $P_k$ is surjective in $G(\F )$.
    \end{enumerate}

Suppose the residual characteristic of $\F$ divides $k$.  Then
density of $P_k$ on $G(\F )$ implies that $G(\F )$ is a finite extension
of a split unipotent group.  In addition if characteristic of $\F$
is positive, then $G(\F )$ is finite.
\end{theorem}

If the field $\F$ is perfect then for an smooth connected $\F$-group
$G$, $R_{us,\F}(G)= R_{u,\F}(G)=R_u(G)$, the unipotent radical of
$G$, and $G/R_u(G)$ is a reductive $\F$-group. The following
Corollary is an immediate consequence of Theorem \ref{T2}
generalizes Theorem 1.2 of \cite{Ch2} - characteristic zero case:
refer Section 2.2 for details on order of compact groups.

\begin{corollary}
Let $\F$ be non-Archimedean perfect field and $G$ be an algebraic
group over $\F$.  Suppose that the characteristic of $\F$ does not
divide $k$.  Then $P_k$ is dense in $G(\F )$ if and only if $P_k$ is
dense in $(G/R_u(G))(\F )$ if and only if $(G/R_u(G))(\F )$ is
compact and $k$ is co-prime to the order of $(G/R_u(G))(\F )$.
\end{corollary}


Next we consider inheritance of surjectivity of $P_k$ for algebraic
groups: recall that not all (even closed) subgroups inherit
surjectivity or density of $P_k$, e.g., $\Z _p$ (resp. $\Z $) is a
closed subgroup of $\Q _p$ (resp. $\R$) and $P_k$ is surjective on
$\Q _p$ (resp. $\R$) for all $k\geq 1$ but $P_k$ is not even dense
in $\Z _p$ (resp. $\Z $) for any $k$ divisible by $p$ (resp. for any
$k>1$):  compare with Corollary 1.3 of \cite{Ch2}.

\begin{theorem}\label{subgroup}
Let $\F$ be a non-Archimedean local field and $G$ be an algebraic
group defined over $\F$.  Then we have the following:

\begin{enumerate}
\item[(1)] If $P_k$ is surjective on $G(\F)$ and $H$ is an algebraic subgroup
of $G$ defined over $\F$, then $P_k$ is surjective on $H(\F)$.

\item[(2)] If $H$ is a closed (not necessarily algebraic) normal subgroup in
$G(\F)$ and $P_k$ is dense in $H$ as well as in $G(\F )/H$, then
$P_k$ is surjective on $G(\F )$.

\end{enumerate}

In particular, for any algebraic normal subgroup $H$ of $G$ defined
over $\F$, $P_k$ is surjective on $G(\F )$ if and only if $P_k$ is
surjective on both $H(\F )$ and $G(\F )/H(\F )$.

\end{theorem}

We have the following corollary regarding infinitely divisible
algebraic groups over non-Archimedean local fields.

\begin{corollary}\label{ida}
Let $G$ be an algebraic group over a non-Archimedean local field.
Suppose $P_k$ is surjective on $ G(\F)$ for all $k\in\N$.  Then
$G(\F )$ is unipotent.  In addition if characteristic of $\F$ is
positive, then $G(\F )=\{ e \}$.
\end{corollary}

The above results are proved using canonical realization of
algebraic groups as subgroups of matrix groups.  We note that apart
from algebraic groups, there are some other groups that have
interesting and enough linear representations.  Therefore, we
consider linear representation of groups and prove the results for
general groups in terms of their linear representations.

Recall that a linear representation of a group $G$ over a local
field $\F$ is a continuous homomorphism $\rho \colon G \to {\rm GL}(n ,
\F)$.

Apart from algebraic groups, Lie groups over local fields is an
interesting class that admit a linear (not necessarily injective)
representation, namely the adjoint representation: recall that ${\rm
Ad}$ is the adjoint representation of $G$ defined on the Lie algebra
of $G$ (see \cite{Bo} and \cite{Se} for more information on Lie
groups).

Various classes of $p$-adic Lie groups were introduced using the
adjoint representation and interesting results were obtained (ref.
\cite{Rr} and \cite{Ra}).  Motivated by these studies we now
introduce the following:  let $\rho \colon G \to {\rm GL}(n, \F)$ be a
linear representation of $G$.

We say that $G$ is called $\rho$-type $R$ if eigenvalues of
$\rho(g)$ are of absolute value $1$ for all $g\in G$ and $G$ is
called $\rho$-unipotent (resp., $\rho $-compact) if ${\rm \rho}(G)$
is contained in an unipotent (resp., compact) subgroup of ${\rm GL}(n , \F
)$.

We now give the results for Lie groups.

\begin{theorem}\label{lie} Let $G$ be a Lie group over a non-Archimedean
local field $\F$ and $P_k$ is dense in $G$ for $k>1$. Then we have
the following:

\begin{enumerate}

\item[(1)] $G$ is type $R$.

\item[(2)] ${\rm Ad}(G)$ is contained in a compact extension of an
unipotent normal subgroup.

\item[(3)] If $G$ is compactly generated, then $G$ is Ad-compact.

\item[(4)] If the residue characteristic divides $k$ and the characteristic of
of $\F$  is zero, then $\overline {{\rm Ad }(G)}$ is a finite
extension of an unipotent group.

\item[(5)] If the characteristic $p>0$ divides $k$, ${\rm Ad }(G)$ is finite.

\item[(6)] If $P_k$ is dense in $G$ for all $k\geq 1$, then ${\rm Ad}~(G)$ is
a $\F$-split unipotent group, in particular, $G$ is ${\rm
Ad}$-unipotent.  In addition if the characteristic of $\F$ is
positive, then {\rm Ad} is trivial.

\end{enumerate}
\end{theorem}

\section{Preliminaries}

\subsection{Semi-direct product}
Let $H$ and $N$ be locally compact groups.  We say that $H$ acts on
$N$ by automorphisms if there is a homomorphism $\phi \colon H \to
{\rm Aut}~(N)$ such that the map $(g, x)\mapsto \phi (g)(x)$ is
continuous.  In this case we define the semi-direct product
$H\ltimes N$ of $H$ and $N$ as the product space $H\times N$ with
binary operation:

$$(g,x) (h,y) = (gh, xg(y))$$ for all $g, h \in G$ and $x, y \in X$.

Then $H\ltimes N$ is a locally compact group.  Identifying $g\in H$
with $(g,e)\in H\ltimes N$ and $x\in N$ with $(e,x) \in H\ltimes N$,
we may view $H$ and $N$ as closed subgroups of $H\ltimes N$.  This
in particular implies that $N$ is a normal subgroup of $H\ltimes N$.
Semi-direct product is a useful technique, particularly helps us to
prove Lemma \ref{ga}.

\begin{example}
\begin{enumerate}

\item Any closed subgroup of ${\rm GL}(n , \F )$ acts on $\F ^n$ by linear
transformations for a local field $\F$.

\item Generally, any closed subgroup of ${\rm Aut}~(G)$ acts canonically on $G$
by automorphisms.

\item Given two closed subgroups $H$ and $N$ of a locally compact group $G$
such that $N$ is normalized by $H$.  Then taking $\phi (g)$ to be
inner automorphism restricted to $N$, defines an action of $H$ on
$N$ by automorphisms.  In this case, if $HN=G$, then $G$ is called
semidirect product of $H$ and $N$.

\end{enumerate}
\end{example}

\subsection{Profinite groups}

A topological group $G$ is said to be profinite if $G$ is a inverse
limit of finite groups.  It is easy to see that any profinite group
is a totally disconnected compact group and the converse is also
true (see Corollary 1.2.4 of \cite{Ws}).

A Steinitz number or supernatural number is a formal infinite
product $\Pi_pp^{n(p)}$, over all primes $p$, where $n(p)$ is a
non-negative integer or infinity.  One may define the product and
l.c.m of super-natural numbers in the natural way: see 2.1 of
\cite{Ws}.

For a finite set $X$, ${\rm Ord}(X)$ denotes the order of $X$ which
is the number of elements in $X$.

For a pro-finite group $G$, the order of $G$ (possibly a
supernatural number) denoted by ${\rm Ord}(G)$ is defined by
$${\rm Ord}(G):= {\rm l.c.m}\{{\rm Ord}(G/U) :{\rm for\; any\; open\; subgroup\;} U\subset G \}.$$
Since $G$ has arbitrarily small compact open normal subgroups, we
may replace open subgroups in the above definition of order, by open
normal subgroups.


We recall the following result from \cite{Se2} and include a simple
proof.

\begin{proposition}\label{F1} Let $G$ be a profinite group.
Then $P_k : G\to G$ is surjective if and only if $k$ is coprime to
${\rm Ord}(G)$, that is for any compact open normal subgroup $U$ of
$G$, $k$ is coprime to the order of the finite group $G/U$.
\end{proposition}

\begin{proof}
It is sufficient to prove the if direction.  Since $G$ is a
profinite group, $G$ has a basis $(U_i)$ of compact open normal
subgroups at $e$.  If $k$ is co-prime to the order of $G/U_i$ for
all $i$, then $P_k$ is surjective $G/U_i$ for all $i$, hence for
$g\in G$ there exist $x_i\in G$ and $u_i\in U_i$ such that $g =
x_i^k u_i$.  Since $(U_i)$ is a basis at $e$, $(gU_i)$ is a basis at
$g$, hence any neighborhood of $g$ contains $x_i^k$ for some $i$.
Thus, $P_k$ is dense in $G$.  Since $G$ is compact, $P_k$ is
surjective on $G$.
\end{proof}

The following lemma is easy to see and known and will be used often.

\begin{lemma}\label{profinite}
Let $G$ be a profinite group and $H$ be a closed subgroup of $G$. If
$P_k$ is surjective on $G$ then $P_k$ is surjective on $H$.
Conversely, if $H$ is normal in $G$ and $P_k$ is surjective on $H$
as well as on $G/H$, then $P_k$ is surjective on $G$.
\end{lemma}

\begin{proof}
As $G$ is profinite, $H$ is also profinite. Note that, Lagrange's
Theorem holds for profinite group (see Proposition 2.1.2 of
\cite{Ws}), so ${\rm Ord}(H)$ divides ${\rm Ord}(G)$. Now, $P_k:G\to
G$ is surjective if and only if $(k, {\rm Ord}(G))=1$. This implies
that $(k,{\rm Ord}(H))=1$, and hence $P_k:H\to H$ is surjective.

Conversely, if $H$ is normal in $G$ and $P_k$ is surjective on both
$H$ and $G/H$, then by Proposition \ref{F1} we get that $k$ is
co-prime to both ${\rm Ord}(H)$ and ${\rm Ord}(G/H)$.  By Lagrange's
Theorem on profinite groups, we get that $k$ is co-prime to ${\rm
Ord}(G)$ (ref. Proposition 2.1.2 of \cite{Ws}).  Another application
of Proposition \ref{F1} proves that $P_k$ is surjective on $G$.
\end{proof}

Next lemma relates order of a profinite group and its open
subgroups.

\begin{lemma}\label{ord}
Let $K$ be a profinite group and $L$ be an open subgroup of $K$.
Then ${\rm Ord}(K) ={\rm l.c.m}\{|K/U|: U \,{\rm is\, open \,
normal\, in\,} L \} =[K:L]{\rm Ord}(L).$
\end{lemma}

\begin{proof}
Note that $L$ is a finite index subgroup in $K$.  For any open
normal subgroup $U$ of $L$, $|K/U| = |L/U|[K:L]$ (see Proposition
2.1.2 of \cite{Ws}).  This proves the second equality.

Let $V$ be any open normal subgroup of $K$.  Then $U= V\cap L$ is an
open normal subgroup of $L$.  Also, $K/V\simeq (K/U) / (V/U)$ and
hence $|K/V| |V/U|= |K/U|$, that is $K/V$ divides $K/U$.  This
proves the first equality.
\end{proof}

\subsection{Algebraic groups}

Let $\F$ be a local field and $G$ be defined over $\F$.  The
$\F$-unipotent (resp., $\F$-split unipotent) radical of $G$ denoted
by $R_{u,\F} (G)$ (resp., $R_{us, \F}(G)$) is defined to be the
maximal connected unipotent normal subgroup of $G$ that is defined
over $\F$ (resp., $\F$-split). These subgroups are contained in the
usual unipotent radical of $G$.  Every connected linear algebraic
$\F$-group $G$ has $R_{us,\F}(G)$. We say that $G$ is {\it
pseudo-reductive} over $\F$ if $R_{u,\F }(G) = 1$ and $G$ is {\it
quasi-reductive} over $\F$ if $R_{us,\F}(G) = 1$.  Since
$R_{us,\F}(G)\subset R_{u,\F}(G)$, every pseudo-reductive $\F$-group
is also quasi-reductive.  If $\F$ is perfect, then any connected
unipotent group defined over $\F$ and is $\F$-split, hence
quasi-reductive $\F$-group and pseudo-reductive are equivalent and
in fact reductive, that is unipotent radical is trivial.  When $\F$
is not perfect, there do exist quasi-reductive $\F$-groups that are
not pseudo-reductive (ref. \cite{So}).

A connected linear algebraic group G over a field of characteristic
$0$ admits a Levi decomposition, that is, it can be written as the
semidirect product of its unipotent radical and a reductive subgroup
known as Levi factor.  In the case the field $\F$ is of positive
characteristic, Levi factors need not exist, even if $\F$ is
algebraically closed (cf. \cite[A.6]{CGP}) the unipotent radical need not
be defined over $\F$.

However, we can get following two short exact sequences.

$$1\to R_{u,\F}(G)\to G\to G/R_{u,\F}(G)\to 1.$$
and  $$1\to R_{us,\F}(G)\to G\to G/R_{us,\F}(G)\to 1.$$

Let $N=R_{us,\F}(G)$ and $M=R_{u,\F}(G)$. Let $P=G/R_{u,\F}(G)$ and
$Q=G/R_{us,\F}(G)$.

The map $G(\F)\to(G/N)(\F)$ is a surjective submersion of $\F$-analytic
manifolds. In particular, this map induces an isomorphism between 
$G(\F)/  R_{us,\F}(G)(\F)$ and $Q(\F)$ as $H^1(\F, R_{us,\F}(G))=0$.

\subsection{$\F$-nilpotent groups}

By an $\F$-nilpotent group, we mean a nilpotent group $N$ such that
if $N=N_0\supset N_1\supset\cdots\supset N_m=\{e\}$ is a central
series of $N$, then each $N_j/N_{j+1}$ is a finite-dimensional
$\F$-vector space. Let $N$ be an $\F$ -nilpotent group and $N_j$ be
a central series. Let $G$ be a group acting on $N$ as a group of
automorphisms of $N$ . The $G$-action on $N$ is said to be $\F$
-linear if the induced action of $G$ on $N_j /N_{j+1}$ is $\F$
-linear for all $j$ .

{\bf Fact 1 (see \cite{DaM}):} Let $G$ be a group and $N$ a normal
subgroup of $G$. Suppose that $N$ is $\F$ -nilpotent with respect to
a field $\F$, and that the conjugation action of $G$ on $N$ is
$\F$-linear. Let $N=N_0\supset N_1\supset\cdots\supset N_r=\{e\}$ be
the central series of $N$. Let $A=G/N$, $x\in G$ and $a=xN\in A$.
Let $k\in\N$ be co-prime to the characteristic of $\F$. Let
$B=\{b\in A|b^k=a\}$ and let $B^*$ be the subset consisting of all
$b$ in $B$ such that for any $j$,  any element of $N_j/N_{j+1}$
which is fixed under the action of $a$ is also fixed under the
action of $b$. Then for any $b\in B^*$, $u\in N$, there exists $y\in
G$ such that $yN=b$ and $y^k=xu$.  But when $G/N$ is a profinite
group, we have the following:

\begin{lemma}\label{cn}
Let $\F$ be a non-archimedean local field and $G$ be a locally
compact group containing a closed normal subgroup $N$ such that $N$
is $\F$-nilpotent and $G/N$ is a profinite group.  Assume that the
conjugation action of $G$ on $N$ is $\F$-linear.  Suppose $k$ is
co-prime to the characteristic of $\F$ and $P_k$ is surjective on
$G/N$.  Then $P_k$ is surjective on $G$.
\end{lemma}

\begin{proof}
Let $x\in G$ and $a=xN$.  Since $P_k$ is surjective on the compact
group $G/N$, Lemma \ref{profinite} implies that $P_k$ is surjective
on $\overline {\left\langle a\right\rangle}$.   Let $N=N_0\supset
N_1\supset\cdots\supset N_r={e}$ be the central series of $N$.  We
observe that, for each $j$, any element of $N_j/N_{j+1}$ fixed by
$a$ is also fixed by the group $\overline {\left\langle
a\right\rangle}$.  In view of Fact 1, we conclude that $x\in
P_k(G)$, that is $P_k$ is surjective on $G$.
\end{proof}

\subsection{Scale function}

Let $G$ be a locally compact totally disconnected group. Then $G$
has arbitrarily small compact open subgroups $U$ of $G$. Let ${\rm
Cos}(G)$ be the set of all compact open subgroup of $G$.

Let ${\rm Aut}(G)$ be the collection of all (continuous)
automorphisms of $G$.  Then the {\bf scale function} $s:{\rm
Aut}(G)\to\N$ is defined as follows: $$s(\alpha):={\rm
 min}\{[\alpha(U):U\cap\alpha(U)]|U\in {\rm Cos}(G)\}$$ for any
$\alpha \in {\rm Aut}(G)$ and the compact open subgroup for which
the minimum is attained is called tidy subgroup of $\alpha$ (see
\cite{Wi1}). The scale function was introduced by G. Willis
\cite{Wi0} and it has proved to be useful.  A property of scale
function that we often uses is the following: $s(\alpha )=1 =
s(\alpha ^{-1})$ if and only if $G$ contains a $\alpha$-invariant
compact open subgroup (ref. Proposition 4.3 of \cite{Wi0}).

For each $x\in G$, let $\alpha _x\colon G \to G$ be the
inner-automorphism defined by $x$, that is $\alpha_x (y)=xyx^{-1}$
for all $y \in G$.  Now define $s(x)=s(\alpha _x)$ and the tidy
subgroup of $\alpha _x$ is defined to be the tidy subgroup of $x$.

It is known that $s(\alpha^n)=s(\alpha)^n$: see \cite{Wi1}.

\section{Representations and Lie groups}

We now prove a dynamic consequence of density of the power map using
semidirect product technique and scale function: recall that $P_k(x)
= x^k$ is the $k$-th power map for $k\geq 1$.

\begin{lemma}\label{ga}
Let $G$ be a locally compact totally disconnected group and $K$ be a
compact normal subgroup of $G$.  Suppose $G$ acts on a totally
disconnected locally compact group $X$ by automorphisms.  Then we
have the following:

\begin{enumerate}

\item If $g\in \overline {P_k(G)K}$ for infinitely many $k$,
then the $g$ action on $X$ fixes a compact open subgroup $K_g$ of
$X$.

\item If $P_k$ is dense in $G/K$ for some $k>1$, then the
$g$ action on $X$ fixes a compact open subgroup $K_g$ of $X$.

\end{enumerate}

\end{lemma}

\begin{proof}
Let $Y = G\ltimes X$ be the semidirect product of $G$ and $X$ for
the given action of $G$ on $X$.  Then $Y$ is a totally disconnected
locally compact group containing $G$ as a closed subgroup and $X$ as
a closed normal subgroup. Let $s$ be the scale function on $Y$. Then
$s$ is continuous on $Y$ (ref. Corollary 4 of \cite{Wi0}).

Let $V$ be a compact open subgroup of $G$ containing $K$ and $W$ be
a compact open subgroup of $X$ fixed by $V$.  Then $V\times W$ is a
compact open subgroup of $Y$ invariant under conjugation by elements
of $K$.

Let $x\in G$.  Then by Lemma 4.2 of \cite{ShW}, the group generated
by $x$ and $K$ has a common tidy subgroup in $Y$. This implies by
Corollary 2.7 of \cite{ShW} that $s(ab)\leq s(a)s(b)$ for all $a$
and $b$ in the group generated by $x$ and $K$ (see also Proposition
7.2 of \cite{GlW}).  Since $s(h)=1$ for all $h\in K$, we have
$s(xh)\leq s(x)\leq s(xh)s(h^{-1})=s(xh).$ Therefore, $s(xh)=s(x)$
for all $h\in K$.

For $g\in G$, let $V_g = \{ x\in G \mid s(x) = s(g) \}$.  Then since
$s$ is continuous on $Y$, $V_g$ is an open neighborhood of $g$ in
$G$ such that $V_g K =V_g$.

Suppose $g\in  \overline {P_k(G)K}$ for infinitely many $k$. Then
there are infinitely many $k$ such that $x_k^k\in V_g$ for some $x_k
\in G$.  This implies that $s(g) =s(x_k)^k$ for infinitely many $k$.
Thus, $s(g)\in\N$ has infinitely many roots, hence $s(g)=1$.
Similarly $s( g^{-1})=1$. Now the first part follows from
Proposition 4.3 of \cite{Wi0}.

Suppose $P_k$ is dense in $G/K$ for some $k>1$.  Let $g\in G$.  Then
there is a $x_1 \in G$ such that $x_1 ^k\in V_g$, hence
$s(x_1)^k=s(g)$.  Now by considering $V_{x_1}$, there is a $x_2 \in
G$ such that $x_2 ^k \in V_{x_1}$.  This implies that $s(x_2)^{k^2}
= s(x_1)^k = s(g)$. Inductively, we get a sequence $(x_n)$ in $G$
such that $s(x_n)^{k^n} = s(g)$. Thus, $s(g)\in\N$ has infinitely
many roots, hence $s(g)=1$. Similarly $s( g^{-1})=1$. Now the second
part follows from Proposition 4.3 of \cite{Wi0}.
\end{proof}

In case, the dynamics in the above Lemma \ref{ga} is linear, then we
can proceed further.

\begin{lemma}\label{lus}
Let $G$ be a locally compact totally disconnected group and $V$ be a
finite-dimensional vector space over a non-Archimedean field $\F$.
Suppose $\rho \colon G \to {\rm GL}(V)$ is the map defining an action of
$G$ on $V$ and $s(x) =1$ for all $x\in G$ where $s$ is the scale
function on $V$.  Then $G$ is $\rho$-type $R$ and there exists a
compact group $K \subset {\rm GL}(V)$ and a $\F$-split unipotent
algebraic group $U\subset {\rm GL}(V)$ normalized by $K$ such that
$K\cap U$ is trivial, $\rho (G)\subset KU$ and $\rho (G)U$ is dense
in $KU$.
\end{lemma}

\begin{proof}
Let $L= \rho (G)$.  Then $s(g) =1$ for all $g\in L$, hence each
$g\in L$ fixes a compact open subgroup of $V$.  Therefore, it
follows that all eigenvalues of $g\in L\subset {\rm GL}(V)$ are
absolute value $1$.  Thus $G$ is $\rho$-type $R$.

Now it follows from Theorem 1 of \cite{CoG} that there exists a flag
$\{ 0 \} =V_0 \subset V_1 \subset \dots \subset V_{m-1} \subset
V_m=V$ such that $L$ on $V_{i}/V_{i-1}$ has only bounded orbits for
any $i\geq 1$. Let $U = \{ \alpha \in {\rm GL}(V) \mid \alpha
(v+V_{i-1}) = v+V_{i-1} ~~{\rm for ~~ all}~~ v\in V_i ~~{\rm and
}~~i\geq 1 \}$ and $K$ be the direct product of closure of the image
of $L$ in ${\rm GL}(V_i/V_{i-1})$. Then $U$ is a split unipotent algebraic
group and $K$ is a compact group that normalizes $U$ such that
$K\cap U$ is trivial and $L\subset KU$ - note that $KU$ is the
semidirect product $K\ltimes U$ of $K$.  This implies that $KU/U
\simeq K$, hence replacing $K\simeq KU/U$ by the closure of
$LU/U\subset KU/U$, we may assume that $LU$ is dense in $KU$.
\end{proof}

We next obtain results in terms of the representation of the groups.

\begin{proposition}\label{nst}
Let $\F$ be a non-Archimedean local field and $H$ be a group with a
linear representation $\rho \colon H\to {\rm GL}(d,\F)$. Suppose
that $P_k$ is dense in $H$ for some $k>1$.  Then we have the
following:

\begin{enumerate}

\item[{\rm (1)}] $H$ is $\rho$-type $R$.

\item[{\rm (2)}] There exists a compact group $K \subset {\rm GL}(d,\F)$ and a
split unipotent algebraic group $U\subset {\rm GL}(d,\F)$ normalized
by $K$ such that $K\cap U$ is trivial, $\rho (H )\subset KU$ and
$\rho (H)U$ is dense in $KU$.  Moreover, $P_k$ is surjective on
$KU/U\simeq K$.

\item[{\rm (3)}] If $k$ is co-prime to the characteristic of $\F$,
then $P_k$ is surjective on $KU$.

\item[{\rm (4)}] If the residual characteristic $p$ of $\F$ divides $k$,
then $K$ is finite, that is $\rho (H)$ is contained in a finite
extension of a split unipotent algebraic group $U$ and $P_k$ is
dense in $\rho (H )\cap U$.

\item[{\rm (5)}] If the residual characteristic $p$ of $\F$ divides $k$ and
the characteristic of $\F$ is zero, then $\overline {\rho (H)}$ is a
finite extension of a split unipotent algebraic group.

\item[{\rm (6)}] If the characteristic $p$ of $\F$ divides $k$, then $\rho (H)$ is
finite.

\item[{\rm (7)}] If $G$ is an $\F$-group and $\rho (H)=G(\F)$, then $G$
has no $\F$-split torus.

\end{enumerate}
\end{proposition}

\begin{remark}
The above results (5) and (6) generalize Corollary 1.7 of \cite{Ch3}
to any linear group (not necessarily algebraic) over any
non-Archimedean local field.
\end{remark}

\noindent{\bf Proof of Proposition~\ref{nst}:}
Let $L= \rho (H)$.  Then $P_k$ is dense in $L$.  By Lemma \ref{ga},
each $g\in L$ fixes a compact open subgroup of $\F^n$. Therefore, it
follows that all eigenvalues of $g\in L\subset {\rm GL}(d,\F)$ are
absolute value $1$.  This proves (1).

It follows from Lemma \ref{ga} that the scale function is trivial on
$L$. Therefore by Lemma \ref{lus}, there exists a compact group $K
\subset {\rm GL}(d,\F)$ and a split unipotent algebraic group
$U\subset {\rm GL}(d,\F)$ normalized by $K$ such that $K\cap U$ is
trivial, $\rho (H )\subset KU$ and $\rho (H)U$ is dense in $KU$.
Since $P_k $ is dense in $L$, $P_k$ is dense in $LU/U$. Since $LU$
is dense in $KU$ and $KU/U$ is compact, we get that $P_k$ is
surjective on $KU/U\simeq K$.  This proves (2).

Suppose $k$ is co-prime to the characteristic of $\F$.  Then since
$P_k$ is surjective on the compact group, $KU/U$,  Lemma \ref{cn}
implies that $P_k$ is surjective on $KU$. This proves (3).

Suppose the residual characteristic $p$ of $\F$ divides $k$.  Since
$K$ is a compact linear group over $\F$, we get that $K$ contains an
open subgroup $K_0$ such that $K_0$ is pro-p group. By Lemma
\ref{profinite}, $P_k$ is surjective on $K_0$. Since $p$ divides
$k$, $K_0$ is trivial.  This implies that $K$ is a finite group and
hence $U$ is an open subgroup of $KU$.  Therefore, $L\cap U$ is open
in $L$.  Since $P_k$ is dense in $L$, $P_k(L)\cap U$ is dense in
$L\cap U$.  Let $g \in P_k(L)\cap U$.  Then there exist $x\in L$
such that $x^k=g\in L\cap U$.  Since $L\subset KU$, there exist
$a\in K$ and $v \in U$ such that $x=av$, hence $x^k= (av)^k= a^kv_k$
where $v_k = \prod _{j=1}^{k} a^{-k+j} va^{k-j} \in U$. This implies
that $a^k\in K\cap U$.  Since $K\cap U$ is trivial, $a^k=e$, hence
$k$ divides the order of $K$ or $a=e$. Since $P_k$ is surjective on
$K$, by Proposition \ref{F1}, we get that $a=e$. Thus, $x= v \in
L\cap U$.  Therefore, $g\in P_k(L\cap U)$.  Thus, $P_k$ is dense in
$L\cap U$.  This proves (4).

Suppose the characteristic of $\F$ is zero and the residual
characteristic $p$ of $\F$ divides $k$.  Since $P_k$ is dense in
$L$, $P_k$ is dense in $\overline L$.  Therefore, replacing $L$ by
$\overline L$, we may assume that $L$ is closed.  Since $L/L\cap
U\simeq LU/U$ is finite, it is sufficient to claim that $L\cap U$ is
a unipotent group.  By (4), $P_k$ is dense in $L\cap U$, we may
assume that $L$ is a closed subgroup of $U$ and $U$ is the smallest
unipotent group containing $L$.  If $U$ is abelian, then $U$ is the
vector space spanned by $L$.  Let $V$ the maximal vector space
contained in $L$.  Then $L/ V$ is compact. Since $P_k$ is dense in
$L$, $P_k$ is surjective on $L/V$ but $p$ divides $k$, hence $L/V$
is finite. Since $L/V$ is a subgroup of the unipotent group $U/V$
which has no elements of finite order, we get that $L=V$, hence
$L=U$.  If $U$ is a general unipotent group, let $Z$ be the center
of $U$. Then since $P_k$ is dense in $L$, $P_k$ dense in
$LZ/Z\subset U/Z$, hence by induction $\overline {LZ}=U$. This
proves that $[U, U]\subset  L$. Using the commutative case for
$U/[U,U]$, we may conclude that $L=U$.  This proves (5).

Suppose characteristic of $\F$ divides $k$.  Since $U$ is a split
unipotent group and $p$ divides $k$, we get that $P_k (U)= \{ e \}$.
Since $P_k$ is dense in $L\cap U$, $L\cap U$ is trivial.  Since $L$
is contained in a finite extension of $U$, $L$ is finite.  This
proves (6).

Suppose $H$ is the group of $\F$-points of an algebraic group $G$
defined over $\F$.  Then the set of eigenvalues of elements of any
non-trivial split torus in $H$ is $\F^*$.  Thus, {\it (1)} implies
that $G$ has no $\F$-split tours.
\qed

It is known that $P_k$ is not dense in any finitely generated
infinite abelian groups: any such group is isomorphic to $F\times \Z
^d$ for $d\geq 1$ for some finite group $F$, hence have $\Z$ as a
quotient. We extend this to compactly generated groups and its
linear representations over non-Archimedean local fields.

\begin{corollary}\label{cg}
Let $\F$ be a non-Archimedean local field and $H$ be a group with a
linear representation $\rho \colon H \to {\rm GL}(d,\F)$.  If $P_k$
is dense in $H$ for some $k>1$ and $H$ is compactly generated then
$H$ is $\rho$-compact.
\end{corollary}

\begin{proof}
Let $L= \rho (H)$.  Then $P_k$ is dense in $L$.  By (2) of
Proposition \ref{nst}, there is a unipotent group $U$ and a compact
linear group $K$ normalizing $U$ such that $L \subset KU$. Let $C$
be a compact generating subset of $L$.  Then $C\subset KM$ where $M$
is a compact subgroup of $U$. Since $K$ normalizes $U$, any compact
subset of $U$ is a contained in a compact $K$-invariant subset of
$U$.  Thus, we may assume that $M$ is a $K$-invariant compact
subgroup.  This implies that $KM$ is a compact subgroup. Since
$C\subset KM$ and $L$ is generated by $C$, we get that $\overline L$
is compact.
\end{proof}

The next result shows that unipotent groups are the only infinite
divisible linear groups.

\begin{corollary}\label{infd}
Let $\F$ be a non-Archimedean local field and $H$ be a group with a
linear representation $\rho \colon H \to {\rm GL}(d,\F)$. Suppose
$P_k$ is dense in $H$ for all $k$.  Then $\overline {\rho (H)}$ is a
split unipotent algebraic group.  In addition if $\F$ has positive
characteristic, $\rho $ is trivial.
\end{corollary}

\begin{proof}
If the characteristic of $\F$ is zero, then by (5) of Proposition
\ref{nst}, $\overline {\rho (H)}$ is a finite extension of an
unipotent algebraic group $U$.  This implies that $\overline {\rho
(H)}/U = (\rho (H)U)/U$. Let $k= |\overline {\rho (H)}/U|$.  Then
since $P_k$ is dense in $H$, $P_k$ is surjective on $\rho (H)U/U=
\overline {\rho (H)}/U$ which has order $k$, hence $\overline {\rho
(H) }/U$ is trivial. Thus, $\overline {\rho (H)}= U$.

If $\F$ has positive characteristic, then by (6) of Proposition
\ref{nst}, $\rho (H)$ is a finite group, hence $\rho (H ) =
\overline {\rho (H)}$. This implies that $U$ is finite, hence
trivial.
\end{proof}

For subgroups of linear groups over global fields: compare with
Section 6 of \cite{Ch2}.

\begin{corollary}
Let $\E$ be a global field and $H$ be a subgroup of ${\rm GL}(d, \E )$.
Assume that $P_k$ is surjective on $H$ for some $k>1$.

\begin{enumerate}

\item[(1)] If the characteristic of $\E$ is $0$, then $H$ contains an
unipotent normal subgroup of finite index.

\item[(2)] If the characteristic of $\E$ is $p>0$, then $H$ is locally finite,
that is any finitely generated subgroup of $H$ is finite.

\item[(3)] If the characteristic $p$ of $\E$ divides $k$, then
$H$ is finite.

\item[(4)] If $P_k$ is surjective on $H$ for all $k\geq 1$, then either
$H$ is a unipotent group or $H$ is trivial depending on
characteristic of $\E$ is $0$ or positive.

\end{enumerate}
\end{corollary}

\begin{proof}
Suppose the characteristic of $\E$ is $0$.  Let $E_p$ be the
$p$-adic completion for $p$ dividing $k$.  Then by (5) of
Proposition \ref{nst}, $H$ is contains an unipotent normal subgroup
of finite index.

Suppose the characteristic of $\E$ is $p>0$.  Then any completion
$E_v$ of $\E$ is non-Archimedean of characteristic $p>0$.  By $(2)$
of Proposition \ref{nst}, $H$ is contained in a compact extension of
a split unipotent group in ${\rm GL}(d, \E_v)$.  If $N$ is a finitely
generated subgroup of $H$, then $N$ is contained in a compact
subgroup of ${\rm GL}(d, \E _v)$.  This implies that $N$ is finite.

Suppose the characteristic $p$ of $\E$ divides $k$.  Then by $(6)$
of Proposition \ref{nst}, in any completion of $\E$, $H$ is finite.
Thus, $H$ is finite.

Last part follows from 1, 3 and Proposition \ref{F1}.
\end{proof}

\noindent{\bf Proof of Theorem \ref{lie}:}  Let $d$ be the dimension
of $G$.  Then ${\rm Ad}(G)$ is a subgroup of ${\rm GL}(d , \F)$.  Thus,
Proposition \ref{nst} and its Corollaries \ref{cg}, \ref{infd}.

\section{Proof of Theorems \ref{T2}, \ref{subgroup} and Corollary \ref{ida}}

We first deal the following:

\begin{lemma}\label{q}
Let $\F$ be a non-Archimedean local field and $H$ be a closed subgroup
of ${\rm GL}(d , \F)$.  Suppose the smallest algebraic group over $\F$ containing
$H$ has no $\F$-split unipotent
normal subgroup.  Then
$P_k$ is dense in $H$ if and only if $H$ is compact and $(k, \Ord
(H)) =1$ (if and only if $P_k$ is surjective on $H$).

\end{lemma}

\begin{proof}
Let $s$ be the scale function on $\F ^d$. Applying Lemma
\ref{ga} to the canonical action of $H$ on $\F ^d$, we get that
$s(x) =1$ for all $x\in H$. Now by Lemma \ref{lus}, there exists a
compact group $K\subset {\rm GL}(d,\F)$ and a split unipotent
algebraic group $U\subset {\rm GL}(d,\F)$ normalized by $K$ such
that $K\cap U$ is trivial, $H \subset KU$ and $HU$ is dense in $KU$.
Let $G$ be the normalizer of $U$.  Then $G$ is defined over $\F$ and
contains $H$.  Thus, the smallest algebraic group, say $G_1$, defined over
$\F$ containing $H$ normalizes $U$.  Since $G_1$ and $U$ are algebraic
groups, $G_1(\F )U$ is closed.  This implies that
$KU\subset G_1(\F )U$.  Since $G_1$ has no $\F$-split unipotent normal
subgroup, $G_1(\F )\cap U$ is trivial, hence the map $f\colon G_1(\F )
\to G_1(\F )U/U$ given by $f(g) = gU$ is an isomorphism of topological
groups.  In particular, $f(H)$ is a closed subgroup.  Since
$HU\subset KU$, $f(H)$ is a closed subgroup of $KU/U$.  Therefore
$f(H)$ is compact. Since $f$ is an isomorphisms, $H$ is compact.
\end{proof}



\noindent{\bf Proof of Theorem~\ref{T2}}: Let $N=R_{us,\F}(G)$ and
$Q=G/R_{us,\F}(G)$.

$(a)\Rightarrow (b):$  
Recall 
that $G(\F )/R_{us, \F}(G)(\F)$ is isomorphic to $Q(\F)$ as $H^1(\F, R_{us,\F}(G))=0$.  Since $P_k$ is 
dense in $G(\F )$, $P_k$ is dense
in $G(\F )/R_{us, \F}(G)(\F)$ and hence by Lemma \ref{q}, we get
that $G(\F )/R_{us, \F}(G)(\F)$ is compact.  Compactness of $G(\F
)/R_{us, \F}(G)(\F)$ implies that $P_k$ is surjective on $G(\F
)/R_{us, \F}(G)(\F)$.

We next observe that $(b)\Rightarrow (c)$ follows from Lemma
\ref{cn} and that $(c) \Rightarrow (a)$ is trivial.

\qed

\noindent{\bf Proof of Theorem~\ref{subgroup}}: Suppose that
$P_k:G(\F)\to G(\F)$ is surjective. Let $N=R_{us,\F}(G)$, and $N_H
=R_{us,\F}(H)$.

Assume that the characteristic of $\F$ does not divide $k$.  By
Theorem \ref{T2}, we have $G(\F)/N(\F)$ is compact and $P_k$ is
surjective on $G(\F)/N(\F )$.  Let $\phi \colon G(\F ) \to G(\F)
/N(\F )$ be the canonical quotient.  Then $\phi (H(\F ))$ is a
closed subgroup of $G(\F)/N(\F)$.  By Lemma \ref{profinite}, we have
$P_k$ is surjective on $\phi (H(\F ))$.  Since ${\rm Ker \phi }\cap
H(\F ) = N(\F )\cap H(\F )\subset N_H(\F)$, we get that $H(\F
)/N_H(\F )$ is a quotient of $\phi (H(\F ))$, hence $H(\F )/N_H(\F
)$ is compact and $P_k$ is surjective on $H(\F)/N_H(\F )$.  Now the
result follows from Theorem \ref{T2}.

Suppose the characteristic of $\F$ divides $k$.  Then $G(\F )$ is
finite, hence the result follows from Lemma \ref{profinite}.

Now let $H$ be a closed normal subgroup of $G(\F)$ and $P_k$ is
dense in $H$ as well as in $G(\F )/H$.

If characteristic of $\F$ divides $k$, then by (6) of Proposition
\ref{nst}, $H$ is finite and $(\Ord (H), k) =1$.  Since $H$ is finite,
$G(\F )/H$ is also linear group, hence (6) of Proposition \ref{nst}
implies that $G(\F )/H$ is also finite and $k$ is co-prime to the
order of $G(\F )/H$.  Thus, $G(\F )$ is finite and $|G(\F ) | = |H||G(\F
)/H|$, hence $k$ is co-prime to the order of $G(\F )$.  Thus, $P_k$
is surjective on $G(\F )$.

We may now assume that the characteristic of $\F$ does not divide
$k$.



Let $Q= G/R_{us, \F}(G)$ and $M=G(\F) /R_{us, \F}(G)(\F )$.  Then $M$
is isomorphic to
$Q(\F)$. So we may assume that $M=Q(\F)$.  Let
$M_1$ be the closure of $HR_{us, \F}(G)(\F )/R_{us, \F}(G)(\F )$.
Then $M_1$ is a closed normal subgroup of $M$.  Since $P_k$ is dense
in $H$, $P_k$ is dense in $M_1$.  Since $Q$ has no $\F$-split
unipotent normal subgroup and $M_1$ is a closed subgroup of $M=Q(\F
)$, Lemma \ref{q} implies that $M_1$ is compact.

Let $d\geq 1$ be such that $M(= Q(\F ))$ is a closed subgroup of $GL(d,
\F)$ and $s$ be the scale function on $\F ^d$.  Since $P_k$ is dense
in $G(\F )/H$, $P_k$ is dense on $M/M_1$. Therefore by Lemma
\ref{ga}, $s$ is trivial on $M$.  By Lemma \ref{lus} implies that there
is a compact subgroup $K\subset {\rm GL}(d, \F)$ and a split unipotent
algebraic group $U\subset {\rm GL}(d, \F)$ normalized by $K$ such that
$K\cap U$ is trivial, $M\subset KU$ and $MU$ is dense in
$KU$.  Since $Q$ is an algebraic group and $M=Q(\F )$, $MU$ is closed, hence
$MU=KU$ and since $Q$ has no $\F$-split unipotent normal
subgroup, $M \cap U$ is trivial.  Therefore, $M \simeq
MU/U = KU/U\simeq K$.  Thus, $M$ is compact.
Therefore, $P_k$ is surjective on the compact groups
$M_1$ and $M/M_1$, hence by Lemma \ref{profinite} we get that $P_k$
is surjective on the compact group $M= G(\F )/R_{us, \F}(G)(\F )$.
Now applying Lemma \ref{cn}, we get that $P_k$ is surjective on
$G(\F )$.

\qed

\noindent{\bf Proof of Corollary \ref{ida}}  Since $G(\F )$ is a
closed subgroup of ${\rm GL}(d, \F)$ for some $d \geq 1$, Corollary
\ref{infd} implies that $G(\F )$ is split unipotent.

If $p >0$ is the characteristic of $\F$, then $P_p$ is surjective on
$R_{us,\F}(G)$ implies that $G(\F ) = R_{us,\F}(G)$ is trivial. \qed

\section{Lattices}
We now consider the situation when the group $G$ has a finite
co-volume subgroup or a co-compact subgroup on which $P_k$ is dense.

\begin{proposition}\label{lat1}
Let $G$ be a totally disconnected locally compact group acting on a
totally disconnected locally compact group $X$ by automorphisms.  Suppose $G$ has a finite co-volume or cocompact subgroup $H$
and $P_k$ is dense in $H$.  Then every element of $G$ fixes a compact open subgroup of $X$.
\end{proposition}

\begin{proof}
Since $P_k$ is dense in $H$,
by Lemma \ref{ga} applied to the conjugate action of $H$ of $G$, we get that every element of $H$ normalizes a compact open subgroup of $G$.  This implies that the modular function of $G$ is trivial on $H$ and $H$
is unimodular.  If $G/H$ is compact, we get that $G$ is unimodular.  Thus, $G/H$ has an invariant measure.
Since $G/H$ is compact, $G/H$ has finite volume.  Thus, we may assume that $H$ has finite co-volume.

Since every element of $H$ normalizes a compact open subgroup of
$G$, by Theorem 2.5 of \cite{Wi2}, every element of $G$ also
normalizes a compact open subgroup of $G$.

Let $g\in G$ and $V$ be a compact open subgroup of $G$ normalized by
$g$.  Then there exists a $n \geq 1$ such that $g^nV\cap H \not = \emptyset$.  Let $h \in H$ be such that $h\in g^nV$.

Let $s$ be the scale function on $X$.  Then by Lemma 4.2 of \cite{ShW}, the group generated
by $g$ and $V$ has a common tidy subgroup in $X$. This implies by
Corollary 2.7 of \cite{ShW} that $s(ab)\leq s(a)s(b)$ for all $a$
and $b$ in the group generated by $g$ and $V$ (see also Proposition
7.2 of \cite{GlW}).  Since $s(a)=1$ for all $a\in V$, we have
$$s(g^na)\leq s(g^n)\leq s(g^na)s(a^{-1})=s(g^na).$$  Therefore, $s(g^na)=s(g^n)$ for all $a\in V$.  Since $h\in g^n V$, $s(h)=s(g^n)$.
It now follows from Lemma \ref{ga} that
$1=s(h)=s(g^n) =s(g)^n$, hence $s(g)=1$.

\end{proof}

It can easily be seen that we can't expect $P_k$ to be dense in $G$
if $P_k$ is dense in a co-compact or a finite co-volume subgroup
$H$. For example, take $H$ to be the trivial subgroup of a compact
group which forces us to ask is this the only obstruction.  In case
$G$ is a group of $\F$-points of an algebraic group defined over
$\F$, we have the following affirmative answer.

\begin{proposition}\label{lat2}
Let $G$ be an algebraic group defined over a non-Archimedean local
field $\F$ and $H$ be a closed subgroup of $G(\F )$ with finite
co-volume or co-compact.  Suppose $P_k$ is dense in $H$.  Then we
have the following:

\begin{enumerate}

\item[(1)] $G(\F )$ is a compact extension of $R_{us, \F}(G)(\F )$.

\item[(2)] If the residual characteristic of $\F$ does not divide $k$,
then $G(\F )$ contains an open subgroup $G_0$ of finite index such
that $P_k$ is surjective on $G_0$ and $G_0$ contains $H$.

\item[(3)] If the residual characteristic of $\F$ divides $k$, then the
characteristic of $\F$ is zero implies that $H$ is a finite
extension of $R_{us, \F}(G)(\F )$ and the characteristic of $\F$ is
positive implies that $G(\F )$ is compact.
\end{enumerate}

\item

\end{proposition}

\begin{proof} Let $M=G(\F ) /R_{us, \F}(G)(\F)$ and
$Q=G/R_{us, \F}(G)$.  Then $M\simeq Q(\F )$, hence we may assume that
$M=Q(\F )$.  Then $M$ is a
closed subgroup of ${\rm GL}(d, \F)$ for some $d\geq 1$.  Let $s$ be the
scale function on $\F ^d$. Then by Proposition \ref{lat1} $s$ is
trivial on $M$.  Now applying Lemma \ref{lus} to $M$, we get that there is
a compact group $K\subset {\rm GL}(d, \F)$ and a split unipotent group
$U\subset {\rm GL}(d, \F)$ normalized by $K$ such that $M\subset KU$
and $MU$ is dense in $KU$.  Since both $Q$ and $U$ are
algebraic, $MU$ is closed, hence $MU = KU$.

Let $f\colon KU\to KU/U$ be the canonical quotient map.  Since $Q$
has no $\F$-split unipotent normal subgroup, $M\cap U$ is
trivial.  Therefore $f$ restricted to $M$ is an isomorphism
onto $KU/U\simeq K$.  Thus, $M$ is compact.

Suppose the residual characteristic $p$ of $\F$ does not divide $k$.
Let $M_0$ be an open normal pro-$p$ subgroup of $M$. Then by
Proposition \ref{F1}, we get that $P_k$ is surjective on $M_0$.  Let
$M_1$ be the closure of $(HR_{us, \F}(G)(\F )/R_{us, \F}(G)(\F ))
M_0$. Then $M_1$ is an open subgroup of finite index in $M$ and by
Lemma \ref{profinite}, $P_k$ is surjective on $M_1$.  Let $G_0$ be
the subgroup of $G(\F )$ containing $R_{us, \F}(G)(\F )$ such that
$G_0/R_{us, \F}(G)(\F )= M_1$.  Then $G_0$ is an open subgroup of
finite index in $G(\F )$. Since the residual characteristic $p$ of
$\F$ does not divide $k$ and  $G_0/R_{us, \F}(G)(\F )$ is compact on
which $P_k$ is surjective, Lemma \ref{cn} implies that $P_k$ is
surjective on $G_0$.

If the residual characteristic $p$ of $\F$ divides $k$ and the
characteristic of $\F$ is zero, then by (5) of Proposition \ref{nst}
we get that $H$ is a finite extension of a split unipotent algebraic
group $V$.  Since $H$ is a finite co-volume or co-compact subgroup
of $G(\F )$, $V$ is also finite co-volume or co-compact subgroup of
$G(\F )$. Since $G(\F )/R_{us, \F}(G)(\F )$ is compact, $V\subset
R_{us, \F}(G) (\F )$.  In particular, $V$ is a finite co-volume
subgroup of $R_{us, \F}(G)(\F )$.  Now, it can be easily shown that
$V= R_{us, \F}(G)(\F)$.  Thus, $H$ is a finite extension of $R_{us,
\F}(G)(\F )$.

If the residual characteristic $p$ of $\F$ divides $k$ and the
characteristic of $\F$ is positive, then by (6) of Proposition
\ref{nst}, $H$ is finite.  This implies that $G(\F )$ itself has
finite volume, hence $G(\F )$ is compact.

\end{proof}

\smallskip
\noindent{\bf Acknowledgements}
We would like to thank Prof. B. Conrad for some helpful suggestions.

\begin{flushleft}
	\begin{flushleft}
		Arunava Mandal and C.R.E. Raja\\
		Theoretical Statistics and Mathematics Unit,\\
		Indian Statistical Institute, Bangalore Centre,\\
		Bengaluru 560059, India.

		E-mail: {\tt a.arunavamandal@gmail.com} and {\tt creraja@isibang.ac.in}
	\end{flushleft}
\end{flushleft}
\end{document}